\documentclass[11pt]{amsart}

\usepackage{amsmath,amssymb,amscd,amsthm,graphicx,enumerate}

\usepackage{hyperref}
\usepackage{color}
\hypersetup{breaklinks=true}

\numberwithin{equation}{section}
\theoremstyle{plain}

\newtheorem{theorem}{Theorem}[section]
\newtheorem{proposition}[theorem]{Proposition}
\newtheorem{lemma}[theorem]{Lemma}
\newtheorem{corollary}[theorem]{Corollary}
\newtheorem{claim}[theorem]{Claim}

\theoremstyle{remark}
\newtheorem{remark}[theorem]{Remark}

\theoremstyle{definition}
\newtheorem{definition}[theorem]{Definition}

\newcommand{\eps}{\varepsilon}

\begin{document}

\title[The blowdown of ancient noncollapsed flows]{The blowdown of ancient noncollapsed mean curvature flows}

\author{Wenkui Du, Robert Haslhofer}

\begin{abstract}
In this paper, we consider ancient noncollapsed mean curvature flows $M_t=\partial K_t\subset \mathbb{R}^{n+1}$ that do not split off a line. It follows from general theory that the blowdown of any time-slice, $\lim_{\lambda \to 0} \lambda K_{t_0}$, is at most $n-1$ dimensional. Here, we show that the blowdown is in fact at most $n-2$ dimensional. Our proof is based on fine cylindrical analysis, which generalizes the fine neck analysis that played a key role in many recent papers. Moreover, we show that in the uniformly $k$-convex case, the blowdown is at most $k-2$ dimensional.
This generalizes the recent results from Choi-Haslhofer-Hershkovits \cite{CHH_wing} to higher dimensions, and also has some applications towards the classification problem for singularities in 3-convex mean curvature flow.
\end{abstract}

\maketitle

\tableofcontents

\section{Introduction}

In the analysis of mean curvature flow it is crucial to understand ancient noncollapsed flows. We recall that a mean curvature flow $M_t$ is called ancient if it is defined for all $t\ll 0$, and noncollapsed if it is mean-convex and there is an $\alpha>0$ so that every point $p\in M_t$ admits interior and exterior balls of radius at least $\alpha/H(p)$, c.f. \cite{ShengWang,Andrews_noncollapsing,HaslhoferKleiner_meanconvex} (in fact, by \cite{Brendle_inscribed,HK_inscribed} one can always take $\alpha=1$). By the work of White \cite{White_size,White_nature,White_subsequent} and by \cite{HH_subsequent} it is known that all blowup limits of mean-convex mean curvature flow are ancient noncollapsed flows. More generally, by Ilmanen's mean-convex neighborhood conjecture \cite{Ilmanen_problems}, which has been proved recently in the case of neck-singularities in \cite{CHH,CHHW}, it is expected even without mean-convexity assumption that all  blowup limits near any cylindrical singularity are ancient noncollapsed flows.

In this paper, we consider the asymptotic structure of ancient noncollapsed flows, specifically the blowdown of any time-slice. To describe this, denote by $M_t=\partial K_t\subset\mathbb{R}^{n+1}$ any ancient noncollapsed mean curvature flow. Recall that such flows are always convex thanks to \cite[Theorem 1.10]{HaslhoferKleiner_meanconvex}. We can assume without essential loss of generality that the solution does not split off a line (otherwise we can reduce the problem to a lower dimensional one). Equivalently, this means that $M_t$ is strictly convex.

\begin{definition}[blowdown]\label{def_blowdown}
Given any time $t_0$, the \emph{blowdown} of $K_{t_0}$ is defined by
\begin{equation}
\check{K}_{t_0}:=\lim_{\lambda\to 0} \lambda\cdot K_{t_0}.
\end{equation}
\end{definition}
It is easy to see that the blowdown exists (in fact it is given as a decreasing intersection of convex sets) and is a cone, i.e. satsifies $\lambda \check{K}_{t_0}= \check{K}_{t_0}$ for all $\lambda >0$.
The blowdown captures important information about the asymptotic structure of the time slices. By general theory  the blowdown is at most $n-1$ dimensional. Indeed, the blowdown cannot be bigger than the singular set in a mean convex flow, which is at most $n-1$ dimensional by the deep theory of White \cite{White_size,White_nature} (see also \cite{HaslhoferKleiner_meanconvex}). Here, we improve this to:

\begin{theorem}[blowdown of ancient noncollapsed flows]\label{thm_main}
Let $M_t=\partial K_t\subset\mathbb{R}^{n+1}$, where $n\geq 3$, be an ancient noncollapsed mean curvature flow that does not split off a line. Then the blowdown of any time slice is at most $n-2$ dimensional, namely
\begin{equation}
\dim \check{K}_{t_0} \leq n-2.
\end{equation}
\end{theorem}
This generalizes the main result from \cite{CHH_wing}, where Choi, Hershkovits and the second author ruled out the potential scenario ancient wing-like noncollapsed flows in $\mathbb{R}^4$. Our theorem shows that the blowdown is always at least $1$ dimension smaller than what one gets from the general bound for the dimension of the singular set.  In particular, it implies:
\begin{corollary}[nonexistence of ancient asymptotically conical flows] There does not exist any ancient noncollapsed solution of the mean curvature flow in $\mathbb{R}^{n+1}$ that is asymptotic to an $(n-1)$-dimensional cone.
\end{corollary}
For comparison, we recall the basic result that there does not exist any nontrivial ancient noncollapsed flow in $\mathbb{R}^{n+1}$ that is asymptotic to an $n$-dimensional cone, as follows easily from the fact that the asymptotic area ratio vanishes \cite{HaslhoferKleiner_meanconvex}. Our corollary rules out solutions that are asymptotic to $n-1$ dimension cones, i.e. solutions asymptotic to cones whose dimension is smaller by $1$ than the dimension of the evolving hypersurfaces, which is much more subtle. It even rules out flows asymptotic to singular cones.

The dimension bound can be improved further in the uniformly $k$-convex case. Recall that an ancient noncollapsed flow is uniformly $k$-convex if there is a constant $\beta>0$ such that $\lambda_1+\ldots +\lambda_k \geq \beta H$, where $\lambda_1\leq \ldots \leq\lambda_n$ denote the principal curvatures. We prove:

\begin{theorem}[blowdown of ancient $k$-convex flows]\label{thm_main2}
If we assume in addition the flow is uniformly $k$-convex, where $k\geq 3$, then the blowdown of any time slice is at most $k-2$ dimensional.
\end{theorem}

Setting $k=3$ and relaxing the strict convexity assumption we obtain:

\begin{corollary}[blowdown of ancient 3-convex flows]\label{cor_3con}
Let $M_t=\partial K_t$ be an ancient noncollapsed uniformly 3-convex mean curvature flow in $\mathbb{R}^{n+1}$. Then its blowdown is
\begin{itemize}
\item either a point (which only happens if the solution is compact),
\item or a halfline,
\item or a line (which only happens for a round shrinking $S^{n-1}\times\mathbb{R}$ or an $(n-1)$-dimensional rotationally symmetric oval times $\mathbb{R}$),
\item or a $2$-dimensional halfplane (which only happens if the solution is a $(n-1)$-dimensional rotationally symmetric bowl soliton times $\mathbb{R}$),
\item or a $2$-dimensional plane (which only happens for a round shrinking $S^{n-2}\times\mathbb{R}^2$),
\item or an $n$-dimensional hyperplane (which only happens for flat $\mathbb{R}^n$).
\end{itemize}
Moreover, in the halfline case, the neutral eigenfunctions are dominant and the diameter of the level sets grows slower than $|t|^{\tfrac{1}{2}+\delta}$ for any $\delta>0$.
\end{corollary}

Let us explain how the corollary follows: If the flow is strictly convex, then its blowdown is a halfline by our theorem (unless the solution is compact in which case the blowdown is simply a point). If the flow is not strictly convex, then after splitting off a line we get an ancient uniformly $2$-convex noncollapsed flow,  and such flows have been classified in the recent breakthrough by Angenent-Daskalopoulos-Sesum \cite{ADS,ADS2} and Brendle-Choi \cite{BC,BC2}. Finally, the `moreover statement' follows as consequence of our proofs as explained in Section  \ref{sec_neutral_mode} and Section \ref{sec_unstable_end}.

As another corollary, for self-similarly translating solutions we obtain:

\begin{corollary}[3-convex translators]\label{cor_3con}
Any noncollapsed uniformly 3-convex translator in $\mathbb{R}^{n+1}$ is $\mathrm{SO}(n-1)$-symmetric.
\end{corollary}

Indeed, any such translator is either an $(n-1)$-dimensional rotationally symmetric bowl soliton times $\mathbb{R}$, or its blowdown is a halfline. The latter in particular implies that the translator has a unique tip point, and hence is $\mathrm{SO}(n-1)$-symmetric by a recent result of Zhu \cite{Zhu_higherdim}.

\begin{remark}[classification problem for 3-convex translators]
Having established the above corollaries it seems likely that the arguments from \cite[Section 3-5]{CHH_translator} generalize to higher dimensions to yield that every noncollapsed uniformly 3-convex translator in $\mathbb{R}^{n+1}$ is either $\mathbb{R}\times\mathrm{Bowl}_{n-1}$, or $\mathrm{Bowl}_n$, or belongs to the one-parameter family of $\mathrm{O}(n-1)\times \mathbb{Z}_2$-symmetric oval-bowls constructed by Hoffman-Ilmanen-Martin-White \cite[Theorem 8.1]{HIMW}.
\end{remark}

Finally, inspecting all the possibilities in Corollary \ref{cor_3con} (blowdown of ancient 3-convex flows), as yet another corollary we obtain:

\begin{corollary}[classification in unstable mode]
The only noncompact ancient noncollpased uniformly $3$-convex flow, with a bubble-sheet tangent flow at $-\infty$, whose unstable mode is dominant is $\mathbb{R}\times\mathrm{Bowl}_{n-1}$.
\end{corollary}

To prove our main theorems, we generalize the arguments from \cite{CHH_wing}, by Choi, Hershkovits and the second author, to higher dimensions as follows:

In Section \ref{sec_coarse}, we collect some basic properties of ancient noncollapsed flows in $\mathbb{R}^{n+1}$. In particular, we show that it is enough to analyze the case where the tangent flow at $-\infty$ is given by
\begin{equation}\label{neck_tangent_intro}
\lim_{\lambda \rightarrow 0} \lambda M_{\lambda^{-2}t}=\mathbb{R}^{\ell}\times S^{n-\ell}(\sqrt{2(n-\ell)t})
\end{equation}
with $\ell=n-1$ (respectively $k-1$).

In Section \ref{sec_fol_barriers}, to facilitate barrier and calibration arguments in later sections, we construct an $\mathrm{O}(\ell)\times \mathrm{O}(n+1-\ell)$ symmetric foliation by rotating and shifting the $d=n+1-\ell$ dimensional shrinker foliation from \cite{ADS}.

In Section \ref{Sec_set_up}, we set up a fine cylindrical analysis, which generalizes the fine neck analysis from \cite{ADS,ADS2,BC,BC2,CHH,CHHW} and the fine bubble sheet analysis from \cite{CHH_wing}. Given any space-time point $X_0=(x_{0}, t_{0})$, we consider the renormalized flow
\begin{equation}
\bar M_\tau = e^{\frac{\tau}{2}} \, \left( M_{t_0-e^{-\tau}} - x_0\right).
\end{equation}
Then, the hypersurfaces $\bar M_\tau$ converge for $\tau\to -\infty$ to the cylinder
\begin{equation}
\Gamma=\mathbb{R}^\ell\times S^{n-\ell}(\sqrt{2(n-\ell)}).
\end{equation}
The analysis over $\Gamma$ is governed by the Ornstein-Uhlenbeck operator
\begin{equation}
\mathcal L=\triangle_{\mathbb{R}^{\ell}}+\frac{1}{2(n-\ell)}\triangle_{S^{n-\ell}}-\frac{1}{2}\sum_{i=1}^{\ell}x_{i}\frac{\partial}{\partial x_{i}}+1.
\end{equation}
Applying the Merle-Zaag lemma, we see that either the unstable eigenfunctions, namely
\begin{equation}
1,x_1,\ldots, x_{n+1},
\end{equation}
are dominant, or the neutral eigenfunctions, namely
\begin{equation}
\{x_i^2-2\}_{i=1,\ldots,\ell} \cup \{ x_i x_j\}_{1\leq i < j \leq \ell} \cup \{ x_i x_J\}_{1\leq i \leq \ell < J\leq n+1}
\end{equation}
are dominant. Further, considering $\tilde{M}_\tau=S(\tau) \bar{M}_\tau$, where $S(\tau)\in \mathrm{SO}(n+1)$ is a suitable rotation, we can kill the contribution from $\{ x_i x_J\}$.

In Section \ref{sec_neutral_mode}, we analyze the case where the neutral eigenfunctions are dominant. We show that along a suitable sequence $\tau_i\to -\infty$, the truncated graphical function $\hat u$ of $\tilde{M}_\tau$ satisfies
\begin{equation}
\lim_{\tau_{i} \to -\infty}\frac{\hat u(\cdot,\tau_{i})}{\|\hat u(\cdot,\tau_{i})\|}= {x'}Qx'^T-2\mathrm{tr}\, Q,
\end{equation}
where $x'=(x_1,\ldots, x_\ell)$, and $Q=\{q_{ij}\}$ is a nontrivial semi-negative definite $\ell\times \ell$-matrix. We then prove the crucial inclusion
\begin{equation}\label{eq_incl_intro}
\check{K}_{t_0}\subseteq \mathrm{ker}\, Q.
\end{equation}
Roughly speaking, directions that are not in the kernel of the quadratic form are directions of inwards quadratic bending, and thus must be compact directions. We make this precise using the Brunn-Minkowski inequality. Once \eqref{eq_incl_intro} is established, we conclude that $\dim \check{K}_{t_0}\leq \ell -1\leq k-2$ for uniformly $k$-convex flows.

In Section \ref{sec_unstable}, we consider ancient noncollapsed flows whose unstable mode is dominant. We prove the fine cylindrical theorem, which says that there exists a nonvanishing vector $(a_1,\ldots, a_{\ell})$ with the following significance: For every space-time center $X$, after suitable recentering in the $\mathbb{R}^{n+1-\ell}$-plane, the truncated graph function $\check u^X(\cdot,\tau)$ of the renormalized flow $\bar{M}_\tau^X$ satisfies
\begin{align}
 \check u^X=e^{\tfrac{\tau}{2}}\left(a_1 x_1 +a_2x_2 +\ldots a_\ell x_\ell\right)+ o(e^{\tfrac{\tau}{2}})
\end{align}
for $\tau\ll 0$ depending only on an upper bound on the cylindrical scale of $X$. To show this, we generalize the proofs of the fine neck theorem from \cite{CHH,CHHW} and the fine bubble sheet theorem from \cite{CHH_wing}.

In Section \ref{sec_unstable_end}, we rule out the case where the unstable mode is dominant. We suppose towards a contradiction that there is an ancient noncollapsed uniformly $k$-convex flow $M_t=\partial K_t$ that does not split off a line, which has a $(k-1)$-dimensional blowdown $\check{K}_{t_0}$. We select two distinct rays of differentiability $R^\pm\subset \partial\check{K}_{t_0}$. Considering suitable corresponding sequences of points in $\partial K_t$, we can pass to limits $K^\pm_t$ that split off $k-2$ lines. Using the classification from Brendle-Choi \cite{BC2}, we argue that $K^\pm_t$ must be $(n+2-k)$-dimensional bowls times $(k-2)$-dimensional planes. However, since these bowls translate in different directions, we obtain a contradiction with the fact that the fine cylindrical vector $(a_1,\ldots, a_\ell)$ is independent of the space-time center point. This concludes the outline of the proof.\\

\noindent\textbf{Acknowledgments.}
This research was supported by the NSERC Discovery Grant and the Sloan Research Fellowship of the second author.

\bigskip

\section{Basic properties of ancient noncollapsed flows}\label{sec_coarse}

Let $M_t=\partial K_t$ be a noncollapsed mean curvature flow in $\mathbb{R}^{n+1}$, where $n\geq 3$, that is noncompact and strictly convex. Assume further that it is uniformly $k$-convex, where $k\geq 3$ (in case $k=n$, there is no further assumption). It is not hard to see that the blowdown
\begin{equation}
\check{K}=\lim_{\lambda\to 0} \lambda K_{t_0}
\end{equation}
is in fact independent of the choice of $t_0$ prior to the extinction time (indeed, as one increases $t_0$, by mean-convexity the blowdown cannot increase, and using suitable spheres as barriers, similarly as in \cite[Section 2]{CHH_wing}, one sees that it cannot decrease either). By \cite{HaslhoferKleiner_meanconvex,CM_uniqueness} the tangent flow of $\{K_t\}$ at $-\infty$ in suitable coordinates, is
\begin{equation}\label{neck_tangent}
\lim_{\lambda \rightarrow 0} \lambda K_{\lambda^{-2}t}=\mathbb{R}^{\ell}\times D^{n+1-\ell}(\sqrt{2(n-\ell)t}),
\end{equation}
for some $\ell\in \{1,\ldots, k-1\}$. Together with mean-convexity it follows that
\begin{equation}
\check{K}\subset \mathbb{R}^{\ell}\times\{0\}.
\end{equation}
If $\ell\leq k-2$, then $\dim \check{K}\leq k-2$ and we are done. Hence, we can assume from now on that $\ell=k-1$. Moreover, since our solution is strictly convex, $\check{K}$ does not contain any line.\\

Given any ancient noncollapsed mean curvature flow $\mathcal{M}$, that is not a round shrinking cylinder, any point $X=(x,t)\in \mathcal{M}$ and scale $r>0$, we consider the flow
\begin{equation}
\mathcal{M}_{X,r}=\mathcal{D}_{1/r}(\mathcal M -X),
\end{equation}
that is obtained from $\mathcal M$ by translating $X$ to the space-time origin and parabolically rescaling by $1/r$.

\begin{definition}[$\eps$-cylindrical]\label{eps_cyl}
We say that $\mathcal M$ is \emph{$\varepsilon$-cylindrical around $X$ at scale $r$}, if $\mathcal{M}_{X,r}$ is $\varepsilon$-close in $C^{\lfloor1/\varepsilon \rfloor}$ in $B(0,1/\varepsilon)\times [-1,-2]$ to the evolution of a round shrinking cylinder $\mathbb{R}^\ell\times S^{n-\ell}(\sqrt{2(n-\ell)t})$ with axis through the origin.
\end{definition}

We fix a small enough parameter $\varepsilon>0$ quantifying the quality of the cylinders for the rest of the paper. Given $X=(x,t)\in\mathcal{M}$, we analyze the solution around $X$ at the diadic scales $r_j=2^j$, where $j\in \mathbb{Z}$. Using Huisken's monotonicity formula \cite{Huisken_monotonicity} and quantitative differentiation (see e.g. \cite{CHN_stratification}), for every $X\in \mathcal{M}$, we can find an integer $J(X)\in\mathbb{Z}$ such that
\begin{equation}\label{eq_thm_quant1}
\textrm{$\mathcal M$ is not $\varepsilon$-cylindrical around $X$ at scale $r_j$ for all $j<J(X)$},
\end{equation}
and
\begin{equation}\label{eq_thm_quant2}
\textrm{$\mathcal M$ is $\tfrac{\varepsilon}{2}$-cylindrical around $X$ at scale $r_j$ for all $j\geq J(X)+N$}.
\end{equation}

\begin{definition}[cylindrical scale]
The \emph{cylindrical scale} of $X\in\mathcal{M}$ is defined by
\begin{equation}
Z(X)=2^{J(X)}.
\end{equation}
\end{definition}

\bigskip

\section{Foliations and barriers}\label{sec_fol_barriers}

We recall from Angenent-Daskalopoulos-Sesum \cite{ADS} that there is some $L_0>1$ such that  for every $a\geq L_0$ and $b>0$, there are $d$-dimensional shrinkers in $\mathbb{R}^{d+1}$,
\begin{align}
{\Sigma}_a &= \{ \textrm{surface of revolution with profile } r=u_a(x_1), 0\leq x_1 \leq a\},\\
\tilde{{\Sigma}}_b &= \{ \textrm{surface of revolution with profile } r=\tilde{u}_b(x_1), 0\leq x_1 <\infty\},\nonumber
\end{align}
as illustrated in \cite[Fig. 1]{ADS}.  Here, the parameter $a$ captures where the  concave functions $u_a$ meet the $x_1$-axis, namely $u_a(a)=0$, and the parameter $b$ is the asymptotic slope of the convex functions $\tilde{u}_b$, namely $\lim_{x_1\to \infty} \tilde{u}_b'(x_1)=b$. We recall:

\begin{lemma}[Lemmas 4.9 and 4.10 in \cite{ADS}] \label{lower_d_ADS}
There exists some $\delta>0$ such that the shrinkers ${\Sigma}_a,\tilde{{\Sigma}}_b$ and ${\Sigma}:=\{x_2^2+\ldots +x_{d+1}^2=2(d-1)\}\subset \mathbb{R}^{d+1}$ foliate the region
\begin{equation}
\{(x_1,\ldots,x_{d+1})\in \mathbb{R}^{d+1}\;|\;  x_1\geq L_0 \textrm{ and }  x_2^2+\ldots +x_{d+1}^2 \leq 2(d-1)+\delta\; \}.
\end{equation}
Moreover, denoting by $\nu_{\textrm{fol}}$ the outward unit normal of this family, we have
\begin{equation}\label{eq_calibration}
\mathrm{div}(e^{-x^2/4} \nu_{\mathrm{fol}})=0.
\end{equation}
\end{lemma}

We now write points in $\mathbb{R}^{n+1}$ in the form $x=(x', x'')$, where $x'=(x_{1},\dots,x_{\ell})$ and $x''=(x_{\ell+1},\dots,x_{n+1})$. We choose $d=n+1-\ell$ and shift and rotate the above foliation to construct a suitable foliation in $\mathbb{R}^{n+1}$:

\begin{definition}[cylindrical foliation]\label{new_fol_def}
For every $a\geq L_{0}$, we denote by $\Gamma_{a}$ the $\mathrm{O}(\ell)\times \mathrm{O}(n+1-\ell)$ symmetric hypersurface in $\mathbb{R}^{n+1}$ given by
\begin{equation}
    \Gamma_{a}=\{(x', x''): (|x'|-1, x'')\in \Sigma_{a}\}.
\end{equation}
Similarly, for every $b\geq L_{0}$, we denote by $\tilde{\Gamma}_{b}$ the
$\mathrm{O}(\ell)\times \mathrm{O}(n+1-\ell)$ symmetric hypersurface in $\mathbb{R}^{n+1}$ given by
\begin{equation}
    \tilde{\Gamma}_{b}=\{(x', x''): (|x'|-1, x'')\in \tilde{\Sigma}_{b}\}.
\end{equation}
\end{definition}

\begin{lemma}[Foliation lemma]\label{foli_lemma}
There exist $\delta>0$ and $L_1>2$ such that the hypersurfaces ${\Gamma}_a$, ${\tilde{\Gamma}}_b$, and the cylinder ${\Gamma}:=\mathbb{R}^{\ell}\times S^{n-\ell}(\sqrt{2(n-\ell)})$ foliate the domain
\begin{equation*}
{\Omega}:=\left\{(x_1,\ldots,x_{n+1}) |  x_1^2+\ldots+x_{\ell}^2\geq L_1^2,\, x_{\ell+1}^2+\ldots + x_{n+1}^2 \leq 2(n-\ell)+\delta\right\}.
\end{equation*}
Moreover, denoting by $\nu_{\mathrm{fol}}$ the outward unit normal of this foliation, we have
\begin{equation}\label{negative_div}
\mathrm{div}(\nu_{\mathrm{fol}}e^{-|x|^2/4})\leq 0\;\;\;\textrm{inside the cylinder},
\end{equation}
and
\begin{equation}\label{positive_div}
\mathrm{div}(\nu_{\mathrm{fol}}e^{-|x|^2/4})\geq 0\;\;\;\textrm{outside the cylinder}.
\end{equation}
\end{lemma}

\begin{proof}
Let $\delta>0$ be as in Lemma \ref{lower_d_ADS}, and set $L_{1}=L_{0}+2(\ell-1)$. The fact that ${\Gamma}_a$, ${\tilde{\Gamma}}_b$ and $\Gamma$ foliate $\Omega$ is implied by Lemma \ref{lower_d_ADS} and Definition \ref{new_fol_def}.\\
Now, observe that for every element in $\Gamma_{*}$ in the foliation of $\Omega$, we have
\begin{equation}
    \mathrm{div}(\nu_{\mathrm{fol}}e^{-|x|^2/4})=(H_{\Gamma_{*}}-\frac{1}{2}\left\langle x, \nu_{\mathrm{fol}} \right\rangle)e^{-|x|^2/4}\, .
\end{equation}
By symmetry, it suffices to compute the curvatures $H_{\Gamma_{*}}$ of $\Gamma_{*}$ in the region $\{x_{1}>0, x_{2}=\dots=x_{\ell}=0\}$,
where we can identify points and unit normals in $\Gamma_{*}$ with the corresponding ones in $\Sigma_{*}$, by disregarding the $x_{2}, \dots, x_{\ell}$ components. The relation between the mean curvature of a surface $\Sigma_{*}$ and its (unshifted) rotation $\Gamma_{*}\subset \mathbb{R}^{n+1}$ on points with $x_{1}>0, x_{2}=\dots=x_{\ell}=0$ is given by
\begin{equation}\label{relation}
    H_{\Gamma_{*}}=H_{\Sigma_{*}}+\frac{\ell-1}{x_{1}}\left\langle e_{1}, \nu_{\mathrm{fol}} \right\rangle,
\end{equation}
where $e_{1}=(1, 0, \dots, 0)\in \mathbb{R}^{n-\ell+2}$.
For the shrinkers $\Sigma_a$, the concavity of $u_{a}$ implies  $\left\langle e_{1}, \nu_{\mathrm{fol}} \right\rangle\geq 0$, so we infer that
\begin{equation}
    H_{\Gamma_{a}}=\frac{1}{2}\left\langle x-e_{1},  \nu_{\mathrm{fol}} \right\rangle+\frac{\ell-1}{x_{1}}\left\langle e_{1}, \nu_{\mathrm{fol}} \right\rangle\leq \frac{1}{2}  \left\langle x, \nu_{\mathrm{fol}} \right\rangle,
\end{equation}
since in $\Omega\cap \{x_{1}>0, x_{2}=\dots=x_{\ell}=0\}$, we have $x_{1}\geq L_{1}\geq 2(\ell-1)$.\\
For the shrinkers $\tilde{\Sigma}_b$, the convexity of $\tilde{u}_{b}$ implies that  $\left\langle e_{1}, \nu_{\mathrm{fol}} \right\rangle\leq 0$, so similarly we have
\begin{equation}
    H_{\tilde{\Gamma}_{b}}\geq \frac{1}{2}  \left\langle x, \nu_{\mathrm{fol}} \right\rangle.
\end{equation}
This proves the lemma.
\end{proof}

\begin{corollary}[Inner Barriers]\label{lemma_inner_barrier}
Let $\{K_{\tau}\}_{\tau\in [\tau_1,\tau_2]}$ be compact domains, whose boundary evolves by renormalized mean curvature flow. If ${\Gamma}_a$ is contained in $K_{\tau}$ for every $\tau\in [\tau_1,\tau_2]$, and $\partial K_{\tau}\cap {\Gamma}_a=\emptyset$ for all $\tau<\tau_2$, then
\begin{equation}
\partial K_{\tau_2}\cap {\Gamma}_a\subseteq \partial {\Gamma}_a.
\end{equation}
\end{corollary}
\begin{proof}
Lemma \ref{foli_lemma} implies that the vector $\vec{H}+\frac{x^{\perp}}{2}$
points outwards of ${\Gamma}_a$. The result now follows from the maximum principle.
\end{proof}

\bigskip

\section{Setting up the fine cylindrical analysis}\label{Sec_set_up}

In this section, we set up the fine cylindrical analysis. This is similar to setting up fine bubble sheet analysis in \cite[Section 4]{CHH_wing}, so we will discuss this briefly.
Throughout this section, let $\mathcal{M}$ be an ancient noncollapsed  flow in $\mathbb{R}^{n+1}$ whose tangent flow at $-\infty$ is a round shrinking cylinder $\mathbb{R}^{\ell}\times S^{n-\ell}(\sqrt{2(n-\ell)t})$. Assume further that $\mathcal{M}$ is not self-similarly shrinking. Given any space-time point $X_0=(x_{0}, t_{0})$, we consider the renormalized flow
\begin{equation}
\bar M^{X_0}_\tau = e^{\frac{\tau}{2}} \, \left( M_{t_0-e^{-\tau}} - x_0\right).
\end{equation}
Then, as $\tau\rightarrow -\infty$, the hypersurfaces $\bar M^{X_0}_\tau$ converges to the cylinder
\begin{equation}
\Gamma=\mathbb{R}^{\ell}\times S^{n-\ell}({\sqrt{2(n-\ell)}}).
\end{equation}
After rescaling, we can assume that $Z(X_0)\leq 1$, and we can find a universal function $\rho(\tau)>0$ with
\begin{equation}\label{univ_fns}
\lim_{\tau \to -\infty} \rho(\tau)=\infty, \quad \textrm{and} -\rho(\tau) \leq \rho'(\tau) \leq 0,
\end{equation}
so that for every $S\in \mathrm{SO}(n+1)$ with $\sphericalangle(S(\Gamma),\Gamma)< \rho(\tau)^{-3}$, the rotated surface $S(M^{X_0}_\tau)$ is the graph of a function $u=u_S(\cdot,\tau)$ over $\Gamma \cap B_{2\rho(\tau)}(0)$ with
\begin{equation}\label{small_graph}
\|u(\cdot,\tau)\|_{C^4(\Gamma \cap B_{2\rho(\tau)}(0))} \leq  \rho(\tau)^{-2}.
\end{equation}
We fix a nonnegative smooth cutoff function $\chi$ satisfying $\chi(s)=1$ for $|s| \leq \frac{1}{2}$ and $\chi(s)=0$ for $|s| \geq 1$. We consider the truncated function
\begin{equation}
\hat{u}(x,\tau):=u(x,\tau) \chi\left(\frac{r}{\rho(\tau)}\right),
\end{equation}
where
\begin{equation}\label{r_def}
r(x):=\sqrt{x_1^2+\ldots +x_{\ell}^2}.
\end{equation}

Using the implicit function theorem, similarly as in \cite[Proposition 4.1]{CHH_wing}, we can find a differentiable function $S^{X_0}(\tau)\in \mathrm{SO}(n+1)$, defined for $\tau$ sufficiently negative, such that for $u:=u_{S^{X_0}(\tau)}$  we have
\begin{equation}\label{eq_orth}
\int_{\Gamma\cap B_{\rho(\tau)}(0)} \langle Ax,\nu_\Gamma\rangle \hat{u}(x,\tau) e^{-\frac{|x|^2}{4}}=0 \qquad\textrm{ for all } A\in o(n+1).
\end{equation}
Moreover, we can arrange that for all $\tau\ll 0$, the matrix
\begin{equation}
A(\tau)=S'(\tau)S(\tau)^{-1}\in o(n+1)
\end{equation}
 satisfies $A_{ij}(\tau)=0$ for $1\leq i,j\leq \ell$, and $A_{IJ}(\tau)=0$ for $\ell+1\leq I,J\leq n+1$.

We now set
\begin{equation}
\tilde{M}_\tau^{X_0}=S^{X_0}(\tau)\bar M_\tau^{X_0},
\end{equation}
where $S^{X_0}(\tau)\in SO(n+1)$ is the fine tuning rotation from above, and let
\begin{equation}
u:=u_{S^{X_0}(\tau)},
\end{equation}
so $u$ describes $\tilde{M}_\tau^{X_0}$ as a graph over the cylinder $\Gamma$.

\begin{proposition}[Inverse Poincare inequality]\label{Gaussian density analysis}
The graph function $u$ satisfies the integral estimates
\begin{equation}
\int_{\Gamma \cap \{|r| \leq L\}} e^{-\frac{|x|^2}{4}} \, |\nabla u(x,\tau)|^2 \leq  C \int_{\Gamma \cap \{|r| \leq \frac{L}{2}\}} e^{-\frac{|x|^2}{4}} \, u(x,\tau)^2
\end{equation}
and
\begin{equation}
\int_{\Gamma \cap \{\frac{L}{2} \leq |r| \leq L\}} e^{-\frac{|x|^2}{4}} \, u(x,\tau)^2 \leq \frac{C}{L^2} \int_{\Gamma \cap \{|r| \leq \frac{L}{2}\}} e^{-\frac{|x|^2}{4}} \, u(x,\tau)^2
\end{equation}
for all $L \in [L_1,\rho(\tau)]$ and $\tau \ll 0$, where $C<\infty$ is a numerical constant.
\end{proposition}

\begin{proof}
The proof is similar to the ones of \cite[Proposition 2.3]{BC} and Proposition \cite[Proposition 4.4]{CHH_wing}, with the only change that we now use the higher dimensional foliation from Definition \ref{new_fol_def} and Lemma \ref{foli_lemma}.
\end{proof}

Let $\mathcal{H}$ be the Hilbert space of all functions $f$ on $\Gamma$ such that
\begin{equation}\label{def_norm}
\|f\|_{\mathcal{H}}^2 = \int_\Gamma \frac{1}{(4\pi)^{n/2}} e^{-\frac{|x|^2}{4}} \, f^2 < \infty.
\end{equation}

\begin{proposition}[evolution equation and error estimate]
\label{Error hat.u-PDE 1}
The truncated graph function $\hat{u}(x,\tau) = u(x,\tau) \, \chi \big ( \frac{r}{\rho(\tau)} \big )$ satisfies the evolution equation
\begin{equation}
\partial_\tau \hat{u} = \mathcal{L} \hat{u}  + \hat{E}+\langle A(\tau)x,\nu_\Gamma\rangle \chi\!\left(\frac{r}{\rho(\tau)}\right),
\end{equation}
where $\mathcal{L}  = \Delta_\Gamma  - \tfrac{1}{2} \,  x^{\text{\rm tan}}\cdot \nabla^\Gamma  +1$,
and where the error term $\hat E$ and the rotation $A=S'S^{-1}$ satisfy the estimates
\begin{equation}
\|\hat{E}\|_{\mathcal{H}} \leq C\rho^{-1} \big( \|\hat{u}\|_{\mathcal{H}}+|A(\tau)|\big)\quad \textrm{and} \quad |A(\tau)|\leq C\rho^{-1}\|\hat{u}\|_{\mathcal{H}}
\end{equation}
for $\tau \ll 0$.
\end{proposition}

\begin{proof}
The proof is similar as in \cite[Section 4]{CHH_wing}, where we now use Proposition \ref{Gaussian density analysis} to control the error terms.
\end{proof}
The cylindrical representation of the operator $\mathcal L$ is explicitly given by
\begin{equation}
\mathcal L=\triangle_{\mathbb{R}^{\ell}}+\frac{1}{2(n-\ell)}\triangle_{S^{n-\ell}}-\frac{1}{2}\sum_{i=1}^{\ell}x_{i}\frac{\partial}{\partial x_{i}}+1
\end{equation}
Analysing the spectrum of $\mathcal L$, the Hilbert space $\mathcal H$ from \eqref{def_norm} can be decomposed as
\begin{equation}
\mathcal H = \mathcal{H}_+\oplus \mathcal{H}_0\oplus \mathcal{H}_-,
\end{equation}
where
\begin{align}
\mathcal{H}_+ =\text{span}\{1, x_1,\ldots, x_{n+1}\},\label{basis_hplus}
\end{align}
and
\begin{align}
\mathcal{H}_0 =\text{span}\left(\{x_i^2-2\}_{1\leq i\leq \ell} \cup \{ x_i x_j\}_{1\leq i < j \leq \ell} \cup \{ x_i x_J\}_{1\leq i \leq \ell < J\leq n+1}\right),\label{basis_hneutral}
\end{align}
and where $x_{1}, \dots, x_{n+1}$ denotes the restriction of the Euclidean coordinate functions to the cylinder $\Gamma$. Moreover, we have
\begin{align}
&\langle \mathcal{L} f,f \rangle_{\mathcal{H}} \geq \frac{1}{2} \, \|f\|_{\mathcal{H}}^2 & \text{\rm for $f \in \mathcal{H}_+$,} \nonumber\\
&\langle \mathcal{L} f,f \rangle_{\mathcal{H}} = 0 & \text{\rm for $f \in \mathcal{H}_0$,} \\
&\langle \mathcal{L} f,f \rangle_{\mathcal{H}} \leq -\frac{1}{n-\ell} \, \|f\|_{\mathcal{H}}^2 & \text{\rm for $f \in \mathcal{H}_-$.} \nonumber
\end{align}
Consider the functions
\begin{align}
&U_+(\tau) := \|P_+ \hat{u}(\cdot,\tau)\|_{\mathcal{H}}^2, \nonumber\\
&U_0(\tau) := \|P_0 \hat{u}(\cdot,\tau)\|_{\mathcal{H}}^2,\label{def_U_PNM} \\
&U_-(\tau) := \|P_- \hat{u}(\cdot,\tau)\|_{\mathcal{H}}^2, \nonumber
\end{align}
where $P_+, P_0, P_-$ denote the orthogonal projections to $\mathcal{H}_+,\mathcal{H}_0,\mathcal{H}_-$, respectively.

\begin{proposition}[Merle-Zaag alternative]\label{mz.ode.fine.neck}
For $\tau\to -\infty$, either the neutral mode is dominant, i.e.
\begin{equation}
U_-+U_+=o(U_0),
\end{equation}
or the unstable mode is dominant, i.e.
\begin{equation}
U_-+U_0\leq C\rho^{-1}U_+.
\end{equation}
\end{proposition}

\begin{proof}
Using Proposition \ref{Error hat.u-PDE 1} (evolution equation and error estimate), we obtain
\begin{align}
&\frac{d}{d\tau} U_+(\tau) \geq U_+(\tau) - C\rho^{-1} \, (U_+(\tau) + U_0(\tau) + U_-(\tau)), \nonumber\\
&\Big | \frac{d}{d\tau} U_0(\tau) \Big | \leq C\rho^{-1} \, (U_+(\tau) + U_0(\tau) + U_-(\tau)), \label{U_PNM_system}\\
&\frac{d}{d\tau} U_-(\tau) \leq -\frac{2}{n-\ell} U_-(\tau) + C\rho^{-1} \, (U_+(\tau) + U_0(\tau) + U_-(\tau)). \nonumber
\end{align}
Hence, the Merle-Zaag ODE lemma \cite{MZ,ChoiMant} implies the assertion.
\end{proof}

\bigskip

\section{Fine cylindrical analysis in the neutral mode}\label{sec_neutral_mode}

In this section, we prove our main theorems in the case where the neutral mode is dominant. Namely, throughout this section we assume that our truncated graph function $\hat u(\cdot,\tau)$ satisfies
\begin{equation}\label{neutral_dom}
U_-+U_+=o(U_0).
\end{equation}
Using this together with the evolution inequality \eqref{U_PNM_system}, given any $\delta>0$, we see that
$ | \tfrac{d}{d\tau} \log U_0 |\leq \tfrac{1}{2}\delta$ for sufficiently large $-\tau$, and thus
\begin{equation}\label{rough decay}
U_0(\tau) \geq e^{\delta\tau}
\end{equation}
for sufficiently large $-\tau$.

\begin{proposition}\label{L^2 convergence lemma}
Every sequence $\{\tau_i\}$ converging to $-\infty$ has a subsequence $\{\tau_{i_m}\}$ such that
\begin{equation}
\lim_{\tau_{i_m} \to -\infty}\frac{\hat u(\cdot,\tau_{i_m})}{\|\hat u(\cdot,\tau_{i_m})\|_{\mathcal{H}}}= {x'}Qx'^T-2\mathrm{tr}\, Q,
\end{equation}
in $\mathcal{H}$-norm, where $Q=\{q_{ij}\}$ is a nontrivial semi-negative definite $\ell\times \ell$-matrix, and $x'=(x_{1}, \dots, x_{\ell})$.
\end{proposition}

\begin{proof}
Recalling \eqref{basis_hneutral} and using the orthogonality condition \eqref{eq_orth}, we see that
\begin{equation}
P_0\hat u \in \text{span}\left(\{x_i^2-2\}_{1\leq i\leq \ell} \cup \{ x_i x_j\}_{1\leq i < j \leq \ell} \right) \subset \mathcal{H}_0.
\end{equation}
Therefore, every sequence $\tau_i \to -\infty$ has a subsequence $\{\tau_{i_m}\}$ such that
\begin{equation}\label{L^2 convergence equation}
\lim_{\tau_{i_m} \to -\infty}\frac{\hat u(\cdot,\tau_{i_m})}{\|\hat u(\cdot,\tau_{i_m})\|_{\mathcal{H}}}=\mathcal{Q}(x')
\end{equation}
with respect to the ${\mathcal{H}}$-norm, where $x'=(x_{1}, \dots, x_{\ell})$, and
\begin{equation}
\mathcal{Q}(x'):= {x'}Qx'^T-2\mathrm{tr}\, Q
\end{equation}
for some nontrivial matrix $Q=\{q_{ij}\}$. After an orthogonal change of coordinates in the $x'$-plane, we can assume that $q_{ij}=0$ for $i\not=j$. Let us show that $q_{11} \leq 0$ (the same argument yields $q_{ii}\leq 0$ for $1\leq i \leq \ell$).\\
We denote by $\tilde{K}_\tau$ the region enclosed by $\tilde{M}_\tau$ and denote by $\mathcal{A}(x',\tau)$ the area of the cross section of $\tilde{K}_\tau$ with fixed $x'$ coordinates.  Explicitly,
\begin{align}\label{Area definition}
\mathcal{A}(x', \tau)=\frac{1}{n+1-\ell}\int_{S^{n-\ell}}\big(\sqrt{2(n-\ell)}+u(x',\omega,\tau)\big)^{n+1-\ell} dS(\omega),
\end{align}
which can be expanded as
\begin{align}\label{Area expansion}
\mathcal{A}(x', \tau)=&|B^{n+1-\ell}(\sqrt{2(n-\ell)})|+\left(2(n-\ell)\right)^{\frac{n-\ell}{2}}\int_{S^{n-\ell}}u(x',\omega, \tau)dS(\omega)\nonumber\\
&+\sum_{k=2}^{n-\ell+1}c_{k}(n,\ell)\int_{S^{n-\ell}}u^{k}(x',\omega, \tau)dS(\omega).
\end{align}
By Brunn's concavity principle, the function $x'\mapsto{\mathcal{A}^{\frac{1}{n-\ell+1}}(x', \tau)}$ is concave.  In particular, we have
\begin{multline}
{\mathcal{A}^{\frac{1}{n-\ell+1}}(x_1-2,x_2,\dots,x_{\ell}, \tau)}+ {\mathcal{A}^{\frac{1}{n-\ell+1}}(x_1+2,x_2,\dots,x_{\ell}, \tau)}\\
\leq 2 {\mathcal{A}^{\frac{1}{n-\ell+1}}(x_1,x_2,\dots,x_{\ell}, \tau)}.
\end{multline}
This implies
 \begin{equation}\label{Brunn-Minkowski}
\int_{[-1,1]^{\ell}}\!\!\!\!\!\!\mathcal{A}^{\frac{1}{n-\ell+1}}(x', \tau)\, dx'\geq \frac{1}{3} \int_{-3}^{3}\int_{-1}^1\ldots \int_{-1}^1 \mathcal{A}^{\frac{1}{n-\ell+1}}(x', \tau)\, dx_1\dots dx_{\ell}.
\end{equation}
Combining this with \eqref{L^2 convergence equation} and \eqref{Area expansion} yields
\begin{equation}
\int_{[-1,1]^{\ell}}\mathcal{Q}(x')\, dx'\geq \frac{1}{3} \int_{-3}^{3}\int_{-1}^1\ldots \int_{-1}^1 \mathcal{Q}(x')\, dx_1\dots dx_{\ell}.
\end{equation}
This implies $q_{11} \leq 0$, and thus concludes the proof.
\end{proof}

\begin{theorem}\label{thm_blowdown_neutral}
If the neutral mode is dominant, then the blowdown $\check K$ satisfies
\begin{equation}
\check{K}\subseteq \mathrm{ker}\, Q.
\end{equation}
for any $Q$ from Proposition  \ref{L^2 convergence lemma}.
In particular, $\dim \check{K}\leq \ell -1\leq k-2$.
\end{theorem}

\begin{proof}
Let
\begin{equation}
\mathcal{Q}(x'):= {x'}Qx'^T-2\mathrm{tr}\, Q.
\end{equation}
Given a unit vector $v$ in the $x'$-plane, let $\{v, w_{1},\dots, w_{\ell-1}\}$ be an orthonormal basis of the $x'$-plane. Set $c=1/n$. Since $\mathcal{Q}(x')>0$ for $|x'|\leq c$, we have
\begin{equation}\label{Q_1_ineq}
\int_{0}^{c}\int_{[0,c]^{\ell-1}}\mathcal{Q}(rv+s_{1}w_{1}+\dots+s_{\ell-1}w_{\ell-1})\, drds>0.
\end{equation}
On the other hand, if $v\notin \ker Q$, then for $d=d(\sphericalangle(v,\ker Q))$ sufficiently large, we have
\begin{equation}\label{Q_2_ineq}
\int_{d}^{d+c} \int_{[0,c]^{\ell-1}}\mathcal{Q}(rv+s_{1}w_{1}+\dots+s_{\ell-1}w_{\ell-1})\, drds<0.
\end{equation}
Defining $\mathcal{A}$ as in \eqref{Area definition}, similarly as in the previous proof we have
\begin{align}\label{lim_Q_2}
\int_{a}^{b} \int_{[0,c]^{\ell-1}}\!\!\!\! \frac{{\scriptstyle \mathcal{A}^{\frac{1}{n-\ell+1}}(rv+s_{1}w_{1}+\dots+s_{\ell-1}w_{\ell-1}, \tau_{i_m})-  |B^{n-\ell+1}(\sqrt{2(n-\ell)})|^\frac{1}{n-\ell+1}}}{||\hat{u}(\cdot, \tau_{i_m})||_{\mathcal{H}}}drds\nonumber\\
\rightarrow c(n,\ell) \int_{a}^{b} \int_{[0,c]^{\ell-1}}\mathcal{Q}(rv+s_{1}w_{1}+\dots+s_{\ell-1}w_{\ell-1})drds
\end{align}
as $m\rightarrow \infty$. Combining \eqref{Q_1_ineq}, \eqref{Q_2_ineq} and \eqref{lim_Q_2} we see that for every $m$ sufficiently large, there exist $r_m,s_{1,m},\dots, s_{\ell-1,m}\in [0,1]$ such that
\begin{multline}\label{A_dec}
\quad \mathcal{A}(r_{m}v+s_{1,m}w_{1}+\dots+s_{\ell-1, ,m}w_{\ell-1},\tau_{i_m}) \\
> \mathcal{A}((r_{m}+d)v+s_{1,m}w_{1}+\dots+s_{\ell-1, ,m}w_{\ell-1},\tau_{i_m}).
\end{multline}
Now, suppose towards a contradiction there is some $\omega\in \check K\setminus \ker Q$. Since $S(\tau)\rightarrow I$ as $\tau\rightarrow -\infty$, for all $-\tau$ sufficiently large, we have
\begin{equation}
\sphericalangle (P(S(\tau)\omega),\ker Q)>\tfrac{1}{2}\sphericalangle(\omega,\ker Q),
\end{equation}
where $P$ denotes the projection to the $x'$-plane. Set
\begin{equation}
v_m:=\frac{P(S(\tau_{i_m})\omega)}{|P(S(\tau_{i_m})\omega)|}.
\end{equation}
Take $v=v_m$ in the previous discussion, and let $\{v_{m}, w_{1,m},\dots, w_{\ell-1,m}\}$ be an orthonromal basis of $x'$-plane. Note that
\begin{equation}\label{stays_insisde}
r_{m}v_{m}+s_{1,m}w_{1,m}+\dots+s_{\ell-1,m}w_{\ell-1,m}+\lambda S(\tau_{i_m})\omega \in \tilde{K}_{\tau_m}\, \forall \lambda \in [0,\infty).
 \end{equation}
On the other hand, by Brunn's concavity principle, the function
\begin{equation}
  r \mapsto  \mathcal{A}^{\frac{1}{n-l+1}}(rv_{m}+s_{1,m}w_{1,m}+\dots+s_{\ell-1,m}w_{\ell-1,m})
\end{equation}
is concave,  as long as it does not vanish. Together with \eqref{A_dec}, this implies that for all $m$ sufficiently large, the area of the cross sections is decreasing for $r>r_m$, and vanishes at some finite $r_\ast$. This contradicts \eqref{stays_insisde}, as the ray would have nowhere to go. This proves the theorem.
\end{proof}

\begin{corollary}
Assume the solution is noncompact and uniformly $3$-convex. Choosing coordinates such that $\check{K}=\{ \lambda e_1|\lambda\geq 0\}$, we have
\begin{equation}
\lim_{\tau \to -\infty}\frac{\hat u(\cdot,\tau)}{\|\hat u(\cdot,\tau)\|_{\mathcal{H}}}= -\frac{x_2^2-2}{\|x_2^2-2\|_{\mathcal{H}}}.
\end{equation}
Furthermore, given any $\delta>0$ for $\tau\ll 0$ the level-sets satisfy
\begin{equation}
\bar{M}_\tau \cap \{x_1=0\} \subseteq B_{\exp{(\delta |\tau|)}}(0).
\end{equation}
\end{corollary}

\begin{proof}
Since a normalized semi-negative definite rank one $2\times 2$ quadratic form is uniquely determined by its kernel, we have convergence without the need of passing to a subsequence. Furthermore, arguing as in  \cite[Corollary 5.4]{CHH_wing} we get the bound for the diameter of the level sets.
\end{proof}

\bigskip

\section{Fine cylindrical analysis in the unstable mode}\label{sec_unstable}

In this section, we prove the fine cylindrical theorem for ancient noncollapsed flows whose tangent flow at $-\infty$ is $\mathbb{R}^{\ell}\times D^{n+1-\ell}(\sqrt{2(n-\ell)t})$. This is similar to the fine neck theorem in \cite{CHH,CHHW} and the fine bubble sheet theorem in \cite{CHH_wing}, so we will discuss this rather briefly.

Throughout  this section, we assume the tilted renormalized flow $\tilde{M}^{X_0}_{\tau}$ around some point $X_0$, has a dominant unstable mode, i.e.
\begin{equation}\label{plus_dom_tilt}
U_0+U_{-} \leq C\rho^{-1} U_{+}.
\end{equation}
Together with \eqref{U_PNM_system}, this implies $\frac{d}{d\tau} U_+(\tau) \geq (1-C\rho^{-1}) U_+$, hence
\begin{equation}\label{U_decay}
U=(1+o(1))U_+ \leq  Ce^{\frac{9}{10}\tau}.
\end{equation}
for sufficiently large $-\tau$. Therefore, Proposition \ref{Error hat.u-PDE 1} yields
\begin{equation}\label{S_decay}
|S^{X_0}(\tau)-I|\leq Ce^{\frac{9}{20}\tau}.
\end{equation}

We now consider the \emph{untilted} renormalized flow $\bar{M}^X_\tau$ around \emph{any} center $X$. For $\tau\leq \mathcal{T}(Z(X))$. we can write $\bar M^X_\tau$ as a graph of a function $\bar u(\cdot,\tau)$ over $\Gamma \cap B_{2\rho(\tau)}(0)$, where $\rho=\rho^X(\cdot,\tau)$ satisfies \eqref{univ_fns} and \eqref{small_graph}. We will work with the truncated function
\begin{equation}
\check u=\bar{u} \chi(r/\rho(\tau)).
\end{equation}
Moreover, we set
\begin{align}
&\check U_+(\tau) := \|P_+ \check{u}(\cdot,\tau)\|_{\mathcal{H}}^2, \nonumber\\
&\check U_0(\tau) := \|P_0 \check{u}(\cdot,\tau)\|_{\mathcal{H}}^2,\\
&\check U_-(\tau) := \|P_- \check{u}(\cdot,\tau)\|_{\mathcal{H}}^2. \nonumber
\end{align}
Applying the Merle-Zaag ODE lemma \cite{MZ,ChoiMant}, we infer that there are universal constants $C,R<\infty$, such that for every $X$ either
\begin{equation}\label{vac_alter}
\check U_{+}+\check U_{-}=o(\check U_0),
\end{equation}
or
\begin{equation}\label{plus_untilt_dom}
\check U_0+\check U_{-} \leq C\rho^{-1} \check U_{+}
\end{equation}
whenever $\rho\geq R$.
\begin{lemma}[dominant mode]\label{lemma_dominant} The unstable mode is dominant for the untilted renormalized flow  $\bar{M}^{X}_{\tau}$ with any center $X$, i.e. \eqref{plus_untilt_dom} holds for all $X$.
\end{lemma}

\begin{proof}
Let us first show that the statement holds for $X=X_0$. Combining \eqref{U_decay} and \eqref{S_decay}, we infer that the untilted flow satisfies
\begin{equation}\label{untilt_decay}
\check  U \leq Ce^{\frac{9}{20}\tau}.
\end{equation}
If we had $\check U_{+}+\check U_{-}=o(\check U_0)$, then arguing as at the beginning of Section \ref{sec_neutral_mode}  we would see that $\check U\geq e^{\delta \tau}$ for every $\delta>0$ and $-\tau$ sufficiently large, contradicting \eqref{untilt_decay}. Thus, for $X=X_0$, we indeed get \eqref{plus_untilt_dom}.
Finally, since any neck centered at a general point $X$ merges with the neck centered at $X_0$ as $\tau\to -\infty$, the inequality \eqref{plus_untilt_dom} holds for every $X$.
\end{proof}

\begin{proposition}[improved graphical radius]\label{rough barrier}
There exists some $\mathcal{T}=\mathcal{T}(Z(X))>-\infty$, such that
\begin{equation}
\bar \rho(\tau)=e^{-\frac{1}{9}\tau}
\end{equation}
is a graphical radius function satisfying  \eqref{univ_fns} and \eqref{small_graph} for $\tau \leq \mathcal{T}$.
\end{proposition}

\begin{proof}
The proof is similar to the one of \cite[Proposition 6.4]{CHH_wing} (see also \cite{CHH,CHHW}), with the only difference that we now use the higher dimensional barriers from Definition \ref{new_fol_def} and Corollary \ref{lemma_inner_barrier}.
\end{proof}

From now on, we work with $\rho=\bar{\rho}$ from Proposition \ref{rough barrier} (improved graphical radius), and in particular, we define $\check u, \check U_0, \check U_\pm, \ldots$ with respect to this improved graphical radius.  Note that the unstable mode is still dominant.

\begin{proposition}[sharp decay estimate]\label{sharp C^10 estimate}
There exist constants $C<\infty$ and $\mathcal{T}>-\infty$, depending only on an upper bound for $Z(X)$, such that for $\tau\leq\mathcal{T}$:
\begin{equation}\label{sharp_hnorm}
\|\check u(\cdot,\tau)\|_{\mathcal{H}} \leq Ce^{\tau/2},
\end{equation}
and
\begin{equation}\label{sharp_c10est}
\|\check u(\cdot,\tau)\|_{C^{n}(\{r\leq 100\})}\leq C e^{\tau/2}.
\end{equation}
\end{proposition}

\begin{proof}
Since $\bar M^X_{\mathcal{T}}$ is an $\varepsilon$-graph over $\Gamma\cap B_{2\rho^X}(0)$, we get $\|\check u(\cdot,\mathcal{T})\|_{\mathcal{H}}\leq   1 $.
Together with the evolution inequality $\tfrac{d}{d\tau} (e^{-\frac{9}{10}\tau}\check U_+) \geq 0$, this implies
\begin{equation}\label{refined L^2 estimate}
\|\check u(\cdot,\tau)\|_{\mathcal{H}} \leq C e^{\frac{9}{20}\tau}
\end{equation}
for $\tau\leq\mathcal{T}$, where $C<\infty$ and $\mathcal{T}>-\infty$ only depend on $Z(X)$.

Now, combining the evolution inequality $\tfrac{d}{d\tau}\check{U}_+\geq (1-C\rho^{-1})\check{U}_+$ with the rough estimate \eqref{refined L^2 estimate} and Proposition \ref{rough barrier} (improved graphical radius) yields
\begin{equation}
\frac{d}{d\tau}\left(e^{-\tau}\check U_+\right) \geq -Ce^{-\tau+\frac{1}{9}\tau+\frac{9}{10}\tau}=-Ce^{\frac{1}{90}\tau}
\end{equation}
for all $\tau\leq \mathcal{T}$.  Thus, $\check U_{+} \leq Ce^{\tau}$. This proves \eqref{sharp_hnorm}. Finally, \eqref{sharp_c10est} follows from the parabolic estimates in \cite[Appendix A]{CHH_wing}.
\end{proof}

\begin{theorem}[fine cylindrical theorem] \label{thm Neck asymptotic} If $\{M_t\}$ is an ancient noncollapsed flow in $\mathbb{R}^{n+1}$, whose  tangent flow at $-\infty$ is  $\mathbb{R}^{\ell}\times S^{n-\ell}(\sqrt{2(n-\ell)t})$, and whose unstable mode is dominant, then there exist constants $a_1,\ldots, a_{n+1}$ with $a_1^2+\ldots +a_\ell^2>0$ with the following significance. For every center point $X=(x_0^1,\ldots x_0^{n+1},t_0)$, the truncated graphical function $\check u^X(\cdot,\tau)$ of the renormalized flow $\bar{M}_\tau^X$
satisfies the fine cylindrical expansion estimates
\begin{equation}\label{expansion 3}
    \Big\|\check{u}^{X}-e^{\frac{\tau}{2}}\Big(\sum_{i=1}^{\ell}a_{i}x_{i}+\sum_{J=\ell+1}^{n+1}a_{J}^Xx_{J}\Big)\Big\|_{\mathcal{H}}\leq Ce^{\frac{5}{9}\tau},
\end{equation}
and
\begin{equation}\label{expansion 4}
    \Big\|\check{u}^{X}-e^{\frac{\tau}{2}}\Big(\sum_{i=1}^{\ell}a_{i}x_{i}+\sum_{J=\ell+1}^{n+1}a_{J}^Xx_{J}\Big)\Big\|_{L^{\infty}(\{r\leq 100\})}\leq Ce^{\frac{19}{36}\tau}
\end{equation}
for $\tau\leq \mathcal{T}$, where $a_J^X=a_J-\frac{1}{\sqrt{2(n-\ell)}}x^{J}_{0}$ for $J\geq \ell+1$,  and where $C<\infty$ and $\mathcal{T}>-\infty$ only depend on an upper bound for the cylindrical scale $Z(X)$.
\end{theorem}

\begin{proof}
In the following $C<\infty$ and $\mathcal{T}>-\infty$ denote constants that might change from line to line, but only depend on an upper bound for the cylindrical scale $Z(X)$.
Recalling the basis of $\mathcal{H}_+$ from \eqref{basis_hplus}, we can write
\begin{align}\label{def P_+check u}
P_+\check u^X=a_0^X(\tau)+\sum_{i=1}^{n+1}a_i^X(\tau)x_i,
\end{align}
where the superscript indicates (a priori) dependence on the center $X$.

Letting $\check E:=(\partial_\tau -\mathcal{L})\check u$, and using $\mathcal L \,1 = 1$, we compute
\begin{align}
\frac{d}{d\tau} a_0^X(\tau)
=a_0^X(\tau)+ {e^{\frac{n-\ell}{2}}(4\pi)^{\frac{-\ell}{2}}}
|S^{n-\ell}(\sqrt{2(n-\ell)})|^{-1} \int_\Gamma\check{E} e^{-\frac{|x|^2}{4}}\, .
\end{align}
Hence, using Proposition \ref{Error hat.u-PDE 1}, Proposition \ref{rough barrier} and Proposition \ref{sharp C^10 estimate}, we obtain
\begin{align}
\Big|\frac{d}{d\tau}\left(e^{-\tau} a_0^X(\tau)\right) \Big| \leq Ce^{-\tau}\|\check E\|_{\mathcal{H}} \leq Ce^{-\tau+\frac{\tau}{2}+\frac{\tau}{9}}
\leq Ce^{-\frac{7}{18}\tau}.
\end{align}
Integrating this from $\tau$ to $\mathcal{T}$ implies
\begin{equation}\label{coeff_est1}
|a_0^X(\tau) |\leq Ce^{\tfrac{11}{18}\tau}\qquad (\tau\leq\mathcal{T}).
\end{equation}
In a similar manner, using $\mathcal{L} x_i=\frac{1}{2}x_i$ for $i=1,\ldots, n+1$, we get
\begin{align}
\Big|\frac{d}{d\tau}\left(e^{-\frac{\tau}{2}} a_i^X(\tau)\right) \Big| \leq Ce^{-\frac{\tau}{2}}\|\check E\|_{\mathcal{H}} \leq Ce^{\frac{\tau}{9}},
\end{align}
so integrating from $-\infty$ to $\tau$ yields
\begin{equation}\label{coeff_est2}
\sum_{i=1}^{n+1}\left|e^{-\tfrac{\tau}{2}} a_i^X(\tau)-\bar{a}_i^X\right| \leq Ce^{\frac{1}{9}\tau} \qquad (\tau\leq\mathcal{T}),
\end{equation}
where
\begin{equation}
\bar{a}_i^X=\lim_{\tau\to -\infty} e^{-\tau/2}a_i^X(\tau)\qquad   (i=1,\ldots, n+1).
\end{equation}

Now, consider the difference
\begin{equation}
D^X:=\check{u}^X-e^{\frac{\tau}{2}}\sum_{i=1}^{n+1}\bar{a}_{i}^Xx_{i}.
\end{equation}
Using  the above, we see that
\begin{equation}\label{point_ineq_D}
|D^X|\leq |\check{u}^X - P_{+}\check{u}^X| +  C (|x|+1) e^{\tfrac{11}{18}\tau}.
\end{equation}
Together with the inequality $\check U_0+\check U_- \leq Ce^{\frac{10}{9}\tau}$, which follows from Lemma \ref{lemma_dominant}, Proposition \ref{rough barrier} and Proposition \ref{sharp C^10 estimate}, this yields
\begin{equation}
\| D^X \|_{\mathcal{H}}\leq Ce^{\frac{5}{9}\tau},
\end{equation}
which proves \eqref{expansion 3} modulo the claim about the coefficients.\\

Next, we observe that
\begin{align}
\|D^X \|_{L^2( \{ r\leq 100\})} \leq C\| D^X \|_{\mathcal{H}} \leq Ce^{\frac{5}{9}\tau},
\end{align}
and, using Proposition \ref{sharp C^10 estimate}, that
\begin{align}
\|\nabla^{n} D^X \|_{L^2( \{ r\leq 100\})}
\leq C\| \hat{u}^X \|_{C^{n}( \{ r\leq 100\})}+Ce^{\frac{\tau}{2}}  \leq Ce^{\frac{\tau}{2}}.
\end{align}
Applying Agmon's inequality, this yields
\begin{equation}\label{agmon_est}
\|D^X\|_{L^{\infty}( \{ r\leq 100\})}\leq C\|D^X\|_{L^2( \{ r\leq 100\})}^{\frac{1}{2}}\|D^X\|_{H^{n}( \{ r\leq 100\})}^{\frac{1}{2}}\leq C e^{\tfrac{19}{36}\tau}.
\end{equation}

Now, comparing the renormalized flows with center $X=(x_0^1,\ldots,x_0^{n+1},t_0)$ and center $0$, we need to relate the parameters of the functions $\check{u}^{X}$ and $\check{u}^{0}$. Since it is easy to see that the parameters do not depend on $t_0$, we may choose $t_0=0$. We then have
\begin{multline}
\sqrt{2(n-\ell)}+\check{u}^{X}(x'-x_0'e^{\tau/2},x''+O(e^{\tau/2}),\tau)\\
=\mathrm{dist}\left(\left(\sqrt{2(n-\ell)}+\check{u}^{0}(x',x'',\tau)\right)\frac{x''}{|x''|},e^{\tau/2}x_0''\right),
\end{multline}
where we use the notation $x'=(x_1,\ldots, x_\ell)$ and $x''=(x_{\ell+1},\ldots, x_{n+1})$ as before.
By Taylor expansion and Proposition \ref{sharp C^10 estimate}, we have
\begin{multline}
\mathrm{dist}\left(\left(\sqrt{2(n-\ell)}+\check{u}^{0}(x',x'',\tau)\right)\frac{x''}{|x''|},e^{\tau/2}x_0''\right)\\
=\sqrt{2(n-\ell)}+\check{u}^{0}(x', x'',\tau)-e^{\tau/2}  \frac{x''\cdot x''_{0}}{|x''|}+o(e^{\tau/2}).
\end{multline}
Together with \eqref{agmon_est}, the above formulas imply  that  $\bar{a}^{X}_{i}=\bar{a}^{0}_{i}$ for $1\leq i\leq \ell$, and $\bar{a}^{X}_{J}=\bar{a}^{0}_{J}-\frac{1}{\sqrt{2(n-\ell)}}x^{J}_{0}$
for $\ell+1\leq J\leq n+1$.

Finally, by the same contradiction argument as in \cite[Section 6.4]{CHH_wing} (see also \cite[Lemma 5.11]{ADS}) we obtain
\begin{equation}\label{non-vanishing}
a_1^2+\ldots + a_\ell^2>0.
\end{equation}
This finishes the proof of the fine cylindrical theorem.
\end{proof}

\bigskip

\section{Conclusion of the proof}\label{sec_unstable_end}

In this final section, we conclude the proof of our main theorem:

\begin{theorem}\label{thm_main_restated}
Let $M_t=\partial K_t\subset\mathbb{R}^{n+1}$ be an ancient noncollapsed mean curvature flow that does not split off a line. Asume that the flow is uniformly $k$-convex for some $k\geq 3$. Then the blowdown of any time slice is at most $k-2$ dimensional, namely
\begin{equation}
\dim \check{K}_{t_0}\leq k-2.
\end{equation}
\end{theorem}

\begin{proof} Let $M_t=\partial K_t\subset\mathbb{R}^{n+1}$ be an ancient noncollapsed uniformly $k$-convex mean curvature flow that does not split off a line.
By the reduction from Section \ref{sec_coarse}, we can assume that
\begin{equation}\label{neck_tangent}
\lim_{\lambda \rightarrow 0} \lambda M_{\lambda^{-2}t}=\mathbb{R}^{\ell}\times S^{n-\ell}(\sqrt{2(n-\ell)t}),
\end{equation}
where $\ell=k-1$. Considering the expansion of the renormalized flow over the cylinder $\Gamma=\mathbb{R}^{\ell}\times S^{n-\ell}(\sqrt{2(n-\ell)})$, by Proposition \ref{mz.ode.fine.neck} (Merle-Zaag alternative) and Lemma \ref{lemma_dominant} (dominant mode), for $\tau\to -\infty$, either the neutral mode dominates for the tilted flow, or the unstable mode dominates for the untilted flow with respect to any center. In the dominant neutral mode case, we have shown in Theorem \ref{thm_blowdown_neutral} that the blowdown is at most $k-2$ dimensional. We can thus assume that we are in the dominant unstable mode case. Since we already know that $\dim \check K\leq \ell$, and since the dimension of convex sets is always an integer, it suffices to rule out the scenario that $\dim \check K= \ell=k-1$.

Suppose towards a contradiction that $\dim \check K= \ell=k-1$. Recall that $\check K$ does not contain any line and also recall that $k\geq 3$. Choosing suitable coordinates, we can assume that $\check{K}\setminus\{0\}$ is contained in $(\mathbb{R}^{\ell}\times \{0 \})\cap \{x_1> 0\}$ and contains the positive $x_1$-axis in its interior.

\begin{claim}[points of differentiability]\label{claim_differentiable}
There exists a two-dimensional plane $P\subseteq \mathbb{R}^\ell$ containing the $x_1$-axis, such that all points $0\neq p\in \partial \check K \cap P$ are twice differentiable points of $\partial \check K\subseteq \mathbb{R}^\ell$.
\end{claim}

\begin{proof}[{Proof of Claim \ref{claim_differentiable}}]
Observe first that since $\check K$ is a cone, whenever there is some point $0\neq p\in \partial \check K \cap P$ that is not twice differentiable, then $\partial\check K \cap P$ in fact contains a $1$-dimensional ray of points that are not twice differentiable.

Now suppose towards a contradiction that for every two-plane $P\subseteq \mathbb{R}^\ell$ containing the $x_1$-axis, there is some point $0\neq p\in \partial \check K \cap P$ that is not twice differentiable. Since the space of two-planes containing the $x_1$-axis is $(\ell-2)$-dimensional, together with the above it follows that the set of points in $\partial \check K$ that are not twice differentiable has positive $(\ell-1)$-dimensional Hausdorff measure. This contradicts Alexandrov's theorem.
\end{proof}

Rotating coordinates, we can assume that $P$ is the $x_1x_2$-plane. Observe that $\partial \check K\cap P$ consists of two rays $R^\pm$ with directional vectors $e_{\pm}$ satisfying $e_{+}\cdot e_{2}>0$ and $e_{-}\cdot e_{2}<0$, respectively.
By a space-time translation, we may also assume that $t_0=0$ and that $0\in M_0$ is the point in $M_0$ with smallest $x_1$-value. Now, for every $h>0$, let $x_h^{\pm}\in M_0\cap \{x_1=h\}\cap P$ be a point which maximizes/minimizes the value of $x_2$ in $K_0\cap \{x_1=h\}\cap P$.

\begin{claim}[cylindrical scale]\label{cyl_scale_bd}
There exists some constant $C<\infty$ such that
\begin{equation}
\sup_h Z(x^{\pm}_h)\leq C.
\end{equation}
\end{claim}

\begin{proof}[{Proof of the Claim \ref{cyl_scale_bd}}]
We will argue similarly as in the proofs of \cite[Claim 7.2]{CHH_wing}, \cite[Proposition 5.8]{CHH} and \cite[Proposition 6.2]{CHHW}.\\
Suppose towards a contradiction that $Z(x^{\pm}_{h_i})\to \infty$ for some sequence $h_i\to \infty$. Let $\mathcal{M}^i$ be the sequence of flows obtained by shifting $x^{\pm}_{h_i}$ to the origin, and parabolically rescaling by $Z(x^{\pm}_{h_i})^{-1}$. By \cite[Thm. 1.14]{HaslhoferKleiner_meanconvex} we can pass to a subsequential limit $\mathcal{M}^\infty$, which is an ancient noncollapsed flow that is weakly convex and smooth until it becomes extinct. Note also that $\mathcal{M}^\infty$ has an $\mathbb{R}^\ell \times S^{n-\ell}$ tangent flow at $-\infty$.

Note that $\mathcal{M}^\infty$ cannot be a round shrinking $\mathbb{R}^\ell \times S^{n-\ell}$. Indeed, if such a cylinder became extinct at time $0$, this would contradict the definition of the cylindrical scale, and if it became extinct at some later time, this would contradict the fact that $K^\infty_0\cap P$ is a strict subset of $P$ by construction.

Thus, by Proposition \ref{mz.ode.fine.neck} (Merle-Zaag alternative), for the flow $\mathcal{M}^\infty$, either the neutral mode is dominant or the unstable mode is dominant.
If the neutral mode is dominant, then for large $i$ we obtain a contradiction with the fact that $\mathcal{M}^i$ has dominant unstable mode, using in particular equation \eqref{rough decay}.
If the unstable mode is dominant, then by the fine cylindrical theorem (Theorem \ref{thm Neck asymptotic}), the limit $\mathcal{M}^\infty$ has some nonvanishing fine cylindrical vector $(a_1^\infty,\ldots, a_\ell^\infty)$. However, this contradicts the fact that the fine cylindrical vector of $\mathcal{M}^i$ is obtained from the fine cylindrical vector $(a_1,\ldots,a_\ell)$ of $\mathcal{M}$ by scaling by  $Z(x^{+}_{h_i})^{-1}\to 0$. This concludes the proof of the claim.
\end{proof}

Take $h_i\to \infty$ and consider the sequence $\mathcal{M}^{i,\pm}:=\mathcal{M}-(x_{h_i}^{\pm},0)$. By Claim \ref{cyl_scale_bd} (cylindrical scale), any subsequential limit $\mathcal{M}^{\infty,\pm}$ is an ancient noncollapsed flow with an $\mathbb{R}^\ell \times S^{n-\ell}$ tangent flow at $-\infty$. Moreover, arguing as in the proof of the claim, we see that $\mathcal{M}^{\infty,\pm}$ has a dominant unstable mode, with the same fine cylindrical vector $(a_1,\ldots , a_\ell)$ as our original flow $\mathcal{M}$.

\begin{claim}[splitting off lines]\label{split}
The hypersurfaces $M_0^{\infty,+}$ and $M_0^{\infty,-}$ contain $(\ell-1)$-dimensional planes $P^+\neq P^-$.
\end{claim}

\begin{proof}[{Proof of Claim \ref{split}}]
First observe that since $
{x_{h_i}^\pm }/  {\| x_{h_i}^\pm\|} \to e_{\pm} \in R^\pm$, the hypersurfaces
$M_0^{\infty,+}$ and $M_0^{\infty,-}$ contain a line in direction $e_+\neq e_{-}$. Thus, it remains to find $\ell-2$ lines in nonradial directions.

By the definition of the blowdown, we have the Hausdorff convergence
\begin{equation}
    h_{i}^{-1}(K\cap \{x_{1}=h_{i}\})\rightarrow \check{K}\cap\{x_{1}=1\}.
\end{equation}
This implies
\begin{equation}
 h_{i}^{-1}x^{\pm}_{h_{i}}\rightarrow q^{\pm}\in \partial \check{K}\cap P\cap \{ x_1=1\},
\end{equation}
where $q^{\pm}$ is a twice differentiable point of $\partial \check K$ by Claim \ref{claim_differentiable}.

Now, let $\nu$ be the inwards unit normal of $\check K\cap \{x_1=1\}\subset \mathbb{R}^{\ell-1}$ at $q=q^\pm$, and consider any unit tangent vector $w\in T_{q}(\partial\check K\cap \{x_1=1\})$. Since $\check K$ is twice differentiable at $q$, we have
\begin{equation}
q_\lambda:=q+\lambda w + |\lambda|^{3/2} \nu \in \textrm{Int}(\check K)
\end{equation}
for $0<|\lambda|<\lambda_0$ small enough. By convexity, the segments from $q$ to $q_\lambda$ and $q_{-\lambda}$ are contained in $\check K$. Note also that these segments have length at least $\lambda$ and meet at angle $\pi-\eps(\lambda)$, where $\eps\to 0$ as $\lambda\to 0$. Together with the Hausdorff convergence from above, we see that for $i$ large enough the segments connecting $x_{h_i}$ to $h_iq_{\lambda}$ and $h_iq_{-\lambda}$ are contained in $K$, have length at least $ h_i \lambda$ and meet at angle $\pi-\eps(\lambda)$. Taking $h_i\to \infty$, and sending $\lambda=\lambda_i\to 0$ slowly enough, we conclude that $M_0^{\infty,\pm}$ contains a line in direction $w$. Since $w$ was arbitrary, this completes the proof of the claim.
\end{proof}

By Claim \ref{split} (splitting off lines), we get $M_t^{\infty,\pm}=P^{\pm}\times N_t^{\pm}$, where $N_t^\pm$ is an ancient noncollapsed uniformly $2$-convex flow. Using the classification by Brendle-Choi \cite{BC2}, we infer that $N_t^\pm$ must be a translating bowl soliton. By inspection, we see that the fine cylindrical vector of $P^{\pm}\times N_t^{\pm}$ points in the direction of translation. However, since $N_t^{+}$ and $N_t^{-}$ translate in different directions, this contradicts the fact that $\mathcal{M}^{\infty,+}$ and $\mathcal{M}^{\infty,-}$ have the same fine cylindrical vector. This concludes the proof of the theorem.
\end{proof}

\bigskip

\bibliography{blowdown}

\bibliographystyle{alpha}

\vspace{10mm}

{\sc Wenkui du, Department of Mathematics, University of Toronto,  40 St George Street, Toronto, ON M5S 2E4, Canada}\\

{\sc Robert Haslhofer, Department of Mathematics, University of Toronto,  40 St George Street, Toronto, ON M5S 2E4, Canada}\\

\emph{E-mail:} wenkui.du@mail.utoronto.ca, roberth@math.toronto.edu

\end{document}